\pdfoutput=1 
\documentclass[11pt]{amsart}
\usepackage[utf8]{inputenc}
\usepackage{amssymb,amsmath,amsthm,enumerate,enumitem,colonequals,mlmodern,tikz-cd,microtype}
\usepackage[cal=euler,bfcal,bb=px,bfbb]{mathalpha}
\usepackage[top=3.75cm, bottom=3cm, left=3.5cm, right=3.5cm]{geometry}

\makeatletter
\@namedef{subjclassname@2020}{\textup{2020} Mathematics Subject Classification}
\makeatother

\usepackage{xcolor}
\colorlet{darkblue}{blue!55!black}
\colorlet{darkcyan}{cyan!50!black}
\colorlet{darkgreen}{green!60!black}

\usepackage{hyperref}
\hypersetup{
    colorlinks=true,
    linkcolor= darkblue,
    urlcolor= darkcyan,
    citecolor= darkgreen,
}

\def\eqref#1{\textcolor{darkblue}{(\ref{#1})}}

\PassOptionsToPackage{hyphens}{url}
\usepackage{hyperref}

\usepackage[nameinlink]{cleveref} 
\Crefformat{section}{#2\S#1#3} 
\Crefmultiformat{section}{#2\S\S#1#3}{ and~#2#1#3}{, #2#1#3}{, and~#2#1#3}
\crefname{hypothesis}{hypothesis}{hypotheses}
\Crefname{hypothesis}{Hypothesis}{Hypotheses}

\usepackage[pagewise]{lineno}
\let\oldequation\equation
\let\oldendequation\endequation
\renewenvironment{equation}{\linenomathNonumbers\oldequation}{\oldendequation\endlinenomath}
\expandafter\let\expandafter\oldequationstar\csname equation*\endcsname
\expandafter\let\expandafter\oldendequationstar\csname endequation*\endcsname
\renewenvironment{equation*}{\linenomathNonumbers\oldequationstar}{\oldendequationstar\endlinenomath}
\let\oldalign\align
\let\oldendalign\endalign
\renewenvironment{align}{\linenomathNonumbers\oldalign}{\oldendalign\endlinenomath}
\expandafter\let\expandafter\oldalignstar\csname align*\endcsname
\expandafter\let\expandafter\oldendalignstar\csname endalign*\endcsname
\renewenvironment{align*}{\linenomathNonumbers\oldalignstar}{\oldendalignstar\endlinenomath}


\newcounter{intro}
\newcounter{HypCounter}

\newtheorem{introthm}[intro]{Theorem}

\newtheorem{introcor}[intro]{Corollary}

\newtheorem{introprop}[intro]{Proposition}

\theoremstyle{plain}
\newtheorem{theorem}{Theorem}[section]
\newtheorem{lemma}[theorem]{Lemma}

\newtheorem{proposition}[theorem]{Proposition}

\theoremstyle{definition}

\newtheorem{definition}[theorem]{Definition}
\newtheorem{example}[theorem]{Example}

\newtheorem*{hypothesis*}{Hypothesis}

\newtheorem{remark}[theorem]{Remark}

\newtheorem*{intronot}{Notation}

\newtheorem*{ack}{Acknowledgements}

\numberwithin{equation}{section}
\numberwithin{theorem}{section}

\title[D\'{e}vissage and generation]{D\'{e}vissage for generation \\ in derived categories}

\author[S.~Dey]{Souvik Dey}
\address{S.~Dey,
Faculty of Mathematics and Physics,
Department of Algebra,
Charles University, 
Sokolovsk\'{a} 83, 186 75 Praha, 
Czech Republic}
\email{souvik.dey@matfyz.cuni.cz}

\author[P.~Lank]{Pat Lank}
\address{P.~Lank,
Department of Mathematics,
University of South Carolina, 
Columbia, SC 29208,
U.S.A.}
\email{plankmathematics@gmail.com}

\date{\today}



\begin{document}

\begin{abstract}
    We study a form of d\'{e}vissage for generation in derived categories of Noetherian schemes. First, we extend a result of Takahashi from the affine context to the global setting, showing that the bounded derived category is classically generated by a perfect complex together with structure sheaves of closed subschemes supported on the singular locus. Second, we make an observation for how generation behaves under the derived pushforward of a proper surjective morphism between Noetherian schemes. These results enable us to explicitly identify strong generators for projective schemes with isolated singularities and for singular varieties over a perfect field.
\end{abstract}

\maketitle

\section{Introduction}
\label{sec:intro}

Our work focuses on a form of d\'{e}vissage in the context of generation in derived categories of a Noetherian scheme. This is a process that allows for a better understanding of a collection of objects in such categories by studying a smaller subcollection. Firstly, we identify a subcategory that is constructed from a single perfect complex and the singular locus, which classically generates the bounded derived category of coherent sheaves. Secondly, we make a useful observation regarding the behavior of strong generation under the derived pushforward associated with a proper surjective morphism between Noetherian schemes.

We remind ourselves on a notion of generation for a triangulated category $\mathcal{T}$ that was introduced in \cite{BVdB:2003}. A subcategory $\mathcal{S}$ \textit{classically generates} $\mathcal{T}$ if the smallest thick subcategory containing $\mathcal{S}$, denoted $\langle \mathcal{S} \rangle$, equals $\mathcal{T}$. That is, every object in $\mathcal{T}$ can be built from $\mathcal{S}$ using cones, shifts, finite coproducts, and retracts. For $n \geq 0$, $\langle \mathcal{S} \rangle_n$ denotes the smallest full subcategory generated from $\mathcal{S}$ using these operations with at most $n-1$ cones (see \Cref{def:generators} for details). The subcategory $\mathcal{S}$ \textit{strongly generates} $\mathcal{T}$ if $\langle \mathcal{S} \rangle_n = \mathcal{T}$ for some $n$. When $\mathcal{S}$ consists of a single object $G$, we write $\langle G \rangle$ and $\langle G \rangle_n$. The \textit{Rouquier dimension} of $\mathcal{T}$ is the smallest $n\geq 0$ such that $\langle G \rangle_{n+1} = \mathcal{T}$ for some $G\in \mathcal{T}$; it is set to $+\infty$ otherwise.

The focus will be on the bounded derived category of coherent sheaves over a Noetherian scheme $ X $, denoted $ D^b_{\operatorname{coh}}(X) $. There is a deep connection between classical generation by a single object and the openness of the regular locus in Noetherian schemes \cite{Iyengar/Takahashi:2019,Dey/Lank:2024b}. This tells us the geometry can influence the presence of being generated by a single object. Recent work has developed our understanding of when $ D^b_{\operatorname{coh}}(X) $ can be generated by a single object, with efforts using commutative and noncommutative techniques\footnote{A far from comprehensive list includes: \cite{Aoki:2021,Neeman:2021, Elagin/Lunts/Schnurer:2020, Jatoba:2021, Iyengar/Takahashi:2016,Iyengar/Takahashi:2019, Dey/Lank/Takahashi:2023, Rouquier:2008, BILMP:2023, Bhaduri/Dey/Lank:2023,Burban/Drozd/Gavran:2017, DeDeyn/Lank/ManaliRahul:2024}.}. 

Despite these efforts, a comprehensive picture of generation in these derived categories remains a mystery. To address this, we shift gears by looking at subcategories of $D^b_{\operatorname{coh}}(X)$ which classically generate. The first candidate we had in mind comes from structure sheaves of closed subschemes. 

From \cite[\href{https://stacks.math.columbia.edu/tag/01YD}{Tag 01YF}]{StacksProject}, given a Noetherian scheme $X$, we can infer that every object $\operatorname{coh}(X)$ belongs to the thick subcategory $\mathcal{T}$ generated by $\mathcal{O}_X$ and $i_\ast\mathcal{O}_Z$ where $i \colon Z \to X$ is a closed immersion from an integral scheme\footnote{See proof of \cite[Lemme 6.9.2]{EGAIV2:1965} for variation in the affine case}. There is the distinguished triangle for any $E \in D^b_{\operatorname{coh}}(X)$:
\begin{displaymath}
    Z(E) \to E \to B(E)[1] \to Z(E)[1]
\end{displaymath}
where $Z(E)$ and $B(E)$ denote the cycle and boundary complexes respectively. These complexes decompose into finite direct sums of shifts of coherent sheaves in $\operatorname{coh}(X)$. Then it follows that $\mathcal{T}=D^b_{\operatorname{coh}}(X)$. Moreover, one could tensor $\mathcal{T}$ with any perfect complex $P$ with full support and use projection formula to achieve the same outcome for classical generation. See \Cref{prop:generate_generically_general} for an alternative argument.

There is an interesting refinement for the observation above in the affine setting. Namely, \cite{Takahashi:2014} tells us for $X$ an affine Noetherian scheme that $D^b_{\operatorname{coh}}(X)$ can be classically generated by $\mathcal{O}_X$ and structure sheaves of closed subscheme support in the singular locus of $X$. Recall the singular locus of $X$ is the collection of $p\in X$, denoted by $\operatorname{sing}(X)$, such that $\mathcal{O}_{X,p}$ is not a regular local ring.

The following result tells us how this can globalized.

\begin{introthm}[see \Cref{thm:devissage}]
    \label{introthm:devissage}
    Let $X$ be a Noetherian scheme. Suppose $P$ is a classical generator for $\operatorname{Perf}(X)$. Then $D^b_{\operatorname{coh}}(X)$ is classically generated by $P$ and the objects of the form $\mathbf{R} i_\ast \mathbf{L} i^\ast P$ where $i\colon Z \to X$ is a closed immersion from an integral scheme such that $i(Z)$ is contained in $\operatorname{sing}(X)$.
\end{introthm}

The utility for \Cref{introthm:devissage} is demonstrated when the singular locus is finite as it allows for one to explicitly identify generators. An exciting case is that for Noetherian scheme with isolated singularities (e.g.\ normal quasi-projective surfaces over a field) we have an explicit object that classically generates:

\begin{introcor}
    \label{introcor:isolated_singularities}
    Let $X$ be a scheme with isolated singularities $p_1,\ldots,p_n$ that is projective over a Noetherian local ring $R$. Denote the closed immersion by $t\colon X \to \mathbb{P}^N_R$. Set $\mathcal{L}$ to be the line bundle $t^\ast \mathcal{O}_{\mathbb{P}^n_R} (-1)$. Then $D^b_{\operatorname{coh}}(X)$ is classically generated by the following object:
    \begin{displaymath}
        (\bigoplus^N_{l=0} \mathcal{L}^{\otimes l} ) \oplus (\bigoplus^n_{j=1}i_{j,\ast}\mathcal{O}_{\operatorname{Spec}(\kappa(p_j))})
    \end{displaymath}
    where $i_j\colon \operatorname{Spec}(\kappa(p_j)) \to X$ is the associated closed immersion for each $j$.
\end{introcor}

This is a special case of a more general setting for Noetherian schemes with isolated singularities (see \Cref{ex:surface_generator}). Observe its utility for projective varieties over a field as it identifies an explicit strong generator. To the best of our knowledge, the explicitness of such an object is very new even for mixed or positive characteristic settings. It becomes an interesting problem for future work to now leverage this concrete object in studying the Rouquier dimension of $D^b_{\operatorname{coh}}(X)$ for varieties $X$ over a field with isolated singularities. 

A common method to detecting whether or not $D^b_{\operatorname{coh}}(X)$ can be generated by a single object is typically done by studying the behavior of the derived pushforward. Such strategies has appeared in the setting of varieties with rational singularities \cite[Theorem 1.1]{Hara:2017}, \cite[Corollary 3.3.8]{Favero:2009}, or \cite[Lemma 7.4]{Kawamata:2006}. 

Our next goal was to establish analogous results in a much broader context. Initially, we observed that it is possible to prove that $\langle \mathbf{R}\pi_\ast D^b_{\operatorname{coh}}(Y) \rangle = D^b_{\operatorname{coh}}(X)$ whenever $\pi\colon Y \to X$ is a proper surjective morphism. With the help of new key lemmas, we were then able to adapt an argument of Aoki \cite{Aoki:2021} to strengthen this result, obtaining a statement about strong generation in the derived category. Details can be found in \Cref{lem:boundedness,lem:factoring_n_fold_extensions,lem:truncate_down_to_bounded,lem:weak_form_aoki}. 

We push the upgrade further to say something about generation by a single object.

\begin{introprop}[see \Cref{prop:devissage_classical}]\label{introprop:devissage_classical}
    Let $\pi\colon Y \to X$ be a proper surjective morphism of Noetherian schemes. If $G$ is a classical (resp.\ strong) generator for $D^b_{\operatorname{coh}}(Y)$, then $\mathbf{R}\pi_\ast G$ is a classical (resp.\ strong) generator for $D^b_{\operatorname{coh}}(X)$.
\end{introprop}

\Cref{introprop:devissage_classical} gives generalizations of \cite[Theorem E]{Lank:2023}, \cite[Lemma 3.15]{Lank/Olander:2024}, and \cite[Corollary 8.1.3]{Gaitsgory:2013}. A convenient consequence to \Cref{introprop:devissage_classical} is that one can explicitly write down a strong generator for varieties over a perfect field.

\begin{introcor}
    [see \Cref{ex:generators_variety}]
    Let $X$ be a variety over a perfect field. By \cite[Theorem 4.1]{deJong:1996}, there exists a proper surjective morphism $\pi \colon \widetilde{X}\to X$ with $\widetilde{X}$ is a regular quasi-projective variety. Let $\mathcal{L}$ be a very ample line bundle on $\widetilde{X}$. Then $D^b_{\operatorname{coh}}(X)$ is strongly generated by the  object $\mathbf{R}\pi_\ast (\oplus^{\dim X}_{l=-\dim X} \mathcal{L}^{\otimes l})$.
\end{introcor} 

Similar to \Cref{introcor:isolated_singularities}, we construct an explicit object that strongly generates the bounded derived category of coherent sheaves, $D^b_{\operatorname{coh}}(X)$, for any singular variety $X$ defined over a perfect field. We anticipate that this object will prove to be very useful in future work. Particularly, in the study of the Rouquier dimension for varieties with singularities that are not isolated (i.e.\ where \Cref{introcor:isolated_singularities} is not applicable).

\begin{intronot}
    Let $X$ be a Noetherian scheme.
    \begin{enumerate}
        \item $D(X)$ is the derived category of $\mathcal{O}_X$-modules 
        \item $D_{\operatorname{Qcoh}}(X)$ is the strictly full subcategory of objects in $D(X)$ whose cohomology sheaves are quasi-coherent $\mathcal{O}_X$-modules
        \item $D^b_{\operatorname{coh}}(X)$ is the strictly full subcategory of $D(X)$ whose objects are bounded with coherent cohomology sheaves
        \item $\operatorname{Perf}(X)$ is the strictly full subcategory of $D(X)$ consisting of perfect complexes
        \item $D^b_{\operatorname{coh}}(R)$ denotes, by abuse of notation, $D^b_{\operatorname{coh}}(\operatorname{Spec}(R))$ for a commutative Noetherian ring $R$.
    \end{enumerate}
\end{intronot}

\begin{ack}
    Souvik Dey was partially supported by the Charles University Research Center program No. UNCE/24/SCI/022, and Grant No. GA CR 23-05148S from the Czech Science Foundation. Pat Lank was partially supported by the National Science Foundation under Grant No. DMS-1928930 while in residence at the Simons Laufer Mathematical Sciences Institute (formerly MSRI). Both authors appreciate the helpful comments of anonymous referees that led to improvements in the paper, and thank Adeel A. Khan for making us aware of \cite{Gaitsgory:2013}.
\end{ack}

\section{Generation}
\label{sec:generation}

We briefly cover generation in triangulated categories. For a further treatment on necessary background, the reader is referred to \cite{BVdB:2003, Rouquier:2008}. Let $\mathcal{T}$ be a triangulated category with shift functor $[1]\colon \mathcal{T} \to \mathcal{T}$.

\begin{definition}\label{def:thick}
    Let $\mathcal{S}$ be a subcategory of $\mathcal{T}$.
    \begin{enumerate}
        \item $\operatorname{add}(\mathcal{S})$ is the smallest strictly full subcategory of $\mathcal{T}$ containing $\mathcal{S}$ that is closed under shifts, finite coproducts, and retracts
        \item $\mathcal{S}$ is said to be \textbf{thick} if it is closed under retracts and is a strictly full triangulated subcategory of $\mathcal{T}$
        \item $\langle \mathcal{S} \rangle$ is the smallest thick subcategory in $\mathcal{T}$ containing $\mathcal{S}$
        \item $\langle \mathcal{S} \rangle_0$ is the full subcategory consisting of all objects isomorphic to the zero object
        \item $\langle \mathcal{S} \rangle_1:= \operatorname{add}(\mathcal{S})$
        \item $\langle \mathcal{S} \rangle_n := \operatorname{add} \big( \{ \operatorname{cone}(\varphi) : \varphi\in \operatorname{Hom}_\mathcal{T} (\langle \mathcal{S} \rangle_{n-1}, \langle \mathcal{S} \rangle_1)  \} \big)$.
    \end{enumerate}
    If $\mathcal{S}$ consists of a single object $G$, then we write $\langle G \rangle$ for $\langle \mathcal{S} \rangle$. Additionally, we write $\mathcal{S}\star \mathcal{S}$ for the collection of $E\in \mathcal{T}$ which fit into a distinguished triangle
    \begin{displaymath}
        S_1 \to E \to S_2 \to S_1 [1]
    \end{displaymath}
    where $S_1,S_2\in \mathcal{S}$. For any $n\geq 1$, we set
    \begin{displaymath}
        \mathcal{S}^{\star n} := \underbrace{\mathcal{S}\star \cdots \star \mathcal{S}}_{n \textrm{ times}}.
    \end{displaymath}
\end{definition}

\begin{remark}
    In the notation of \Cref{def:thick}, we have an exhaustive filtration for the smallest thick subcategory containing $\mathcal{S}$ in $\mathcal{T}$:
    \begin{displaymath}
        \langle \mathcal{S} \rangle_0 \subseteq \langle \mathcal{S} \rangle_1 \subseteq \cdots \subseteq \bigcup_{n=0}^\infty \langle \mathcal{S} \rangle_n = \langle \mathcal{S} \rangle.
    \end{displaymath}
\end{remark}

\begin{definition}\label{def:generators}
    Let $E$ be an object of $\mathcal{T}$ and $\mathcal{S}$ be a subcategory of $\mathcal{T}$.
    \begin{enumerate}
        \item $E$ is said to be \textbf{finitely built} by $\mathcal{S}$ if $E$ is in $ \langle \mathcal{S} \rangle$
        \item $\mathcal{S}$ is said to \textbf{classically generate} $\mathcal{T}$ if $\langle \mathcal{S} \rangle = \mathcal{T}$; if $\mathcal{S}$ consists of a single object $G$, then $G$ is called a \textbf{classical generator}
        \item $\mathcal{S}$ is said to \textbf{strongly generate} $\mathcal{T}$ if there exists $n\geq 0$ such that $\langle \mathcal{T} \rangle_n = \mathcal{T}$; if $\mathcal{S}$ consists of a single object $G$, then $G$ is called a \textbf{strong generator}.
    \end{enumerate}
\end{definition}

\begin{example}\label{ex:generators}
    The following list, though far from comprehensive, are familiar instances in both geometric and algebraic contexts where generators can be explicitly identified:
    \begin{enumerate}
        \item $R\oplus k$ is a strong generator for $D^b_{\operatorname{coh}}(R)$ if $(R,\mathfrak{m},k)$ is a local Noetherian ring with an isolated singularity, see \cite[Proposition A.2]{KMVdB:2011}
        \item $\operatorname{Perf}(X)$ admits a classical generator for $X$ a Noetherian scheme, see \cite[Theorem 3.1.1]{BVdB:2003}; additionally, if $X$ is regular and quasi-affine, then $\mathcal{O}_X$ is a classical generator for $D^b_{\operatorname{coh}}(X)$ (see \cite[Corollary 1.10]{BILMP:2023} and \cite[\href{https://stacks.math.columbia.edu/tag/0BQT}{Tag 0BQT}]{StacksProject})
        \item $\bigoplus^{\dim X}_{i=0} \mathcal{L}^{\otimes i}$ is a strong generator for $D^b_{\operatorname{coh}}(X)$ if $\mathcal{L}$ is a very ample line bundle on a smooth quasi-projective variety $X$ over a field and $N\gg 0$, see \cite[Theorem 4]{Orlov:2009}.
        \item $F_\ast^e (\bigoplus^{\dim X}_{i=0} \mathcal{L}^{\otimes i})$ is a strong generator for $D^b_{\operatorname{coh}}(X)$ if $\mathcal{L}$ is a very ample line bundle on a singular quasi-projective variety $X$ over a perfect field of positive characteristic where $F^e \colon X \to X$ is the $e$-th iterate of the Frobenius morphism on $X$, see \cite[Corollary 3.9]{BILMP:2023}.
    \end{enumerate}
\end{example}

\begin{definition}\label{def:support}
    Let $X$ be a Noetherian scheme, $E$ an object of $D^b_{\operatorname{coh}}(X)$, and $i \colon Z \to X$ a closed immersion.
    \begin{enumerate}
        \item $\operatorname{Supp}(E) := \bigcup^{\infty}_{j=-\infty} \operatorname{Supp}(\mathcal{H}^j (E))$ is called the \textbf{support} of $E$
        \item $E$ is \textbf{supported} on $Z$ whenever $\operatorname{Supp}(E)$ is contained in $Z$; we say $E$ has \textbf{full support} if $\operatorname{Supp}(E)=X$.
    \end{enumerate}
\end{definition}

\section{Results}
\label{sec:results}

This section will prove our main results. To start, we need something that is akin to generic freeness of coherent $\mathcal{O}_X$-modules on a Noetherian integral scheme. The following might be known to a few, but we offer an alternative proof for the sake of interest in our setting.

\begin{lemma}\label{lem:derived_generic_free}
    Let $X$ be a Noetherian integral scheme. If $E$ and $G$ are objects of $D^b_{\operatorname{coh}}(X)$ with full support, then there exists an open immersion $i \colon U \to X$ such that $\mathbf{L} i^\ast E $ is in $ \langle \mathbf{L} i^\ast G \rangle_1$ and $\mathbf{L} i^\ast G $ is in $ \langle \mathbf{L} i^\ast E \rangle_1$.
\end{lemma}

\begin{proof}
    Consider a nonempty open affine $U$ in $X$. By \cite[Proposition 3.5]{Letz:2021}, the set $S$ of points $p$ in $U$ such that $E_p $ is in $ \langle G_p \rangle_1$ is Zariski open. Clearly, $S$ is nonempty as it contains the generic point of $X$. By the same reasoning, the set $T$ of points $p$ in $X$ such that $G_p $ is in $ \langle E_p \rangle_1$ is Zariski open and nonempty. Choose a nonempty open affine $W$ in $S\cap T$. Let $j \colon W\to X$ be the associated open immersion. From \cite[Corollary 3.4]{Letz:2021}, it follows that $\mathbf{L} j^\ast E$ is in $ \langle \mathbf{L} j^\ast G \rangle_1$ and $\mathbf{L} j^\ast G $ is in $ \langle \mathbf{L} j^\ast E \rangle_1$ because for each $p$ in $W$, one has $E_p $ is in $ \langle G_p \rangle_1$ and $G_p$ is in $ \langle E_p \rangle_1$.
\end{proof}

\begin{proposition}\label{prop:generate_generically_general}
    Let $X$ be a Noetherian scheme. Suppose $P\in \operatorname{Perf}(X)$ has full support. Then $D_{\operatorname{coh}}^b (X)$ is classically generated by $P$ and objects of the form $\mathbf{R} i_\ast \mathbf{L} i^\ast P$ where $i \colon Z \to X$ is a closed immersion from an integral scheme.
\end{proposition}

\begin{proof}
    Let $\mathcal{S}$ be the collection consisting of $P$ and objects of the form $\mathbf{R} i_\ast \mathbf{L} i^\ast P$ where $i \colon Z \to X$ is a closed immersion from an integral scheme. If $X=\emptyset$, then there is nothing to check, so without loss of generality, $X$ is nonempty. We argue by Noetherian induction. From \cite[Remark 2.23]{Lank:2023}, we may impose that $X$ is reduced and the desired claim holds for any properly contained closed subscheme $Z$ of $X$. Let $E$ be an object of $D^b_{\operatorname{coh}}(X)$. 
    
    If $E$ has support properly contained in $X$, then the induction hypothesis tells us that $E$ belongs to $\mathcal{S}$. Indeed, let $Z$ be the reduced induced closed subscheme structure on $\operatorname{supp}(E)$ with corresponding ideal sheaf $\mathcal{I}$. By \cite[Remark 2.23]{Lank:2023}, there is an integer $n$ such that $\mathcal{I}^n E=(0)$, which ensures $E$ is isomorphic to an object of the form $\mathbf{R} i_\ast E^\prime$ for $E^\prime$ in $D^b_{\operatorname{coh}}(Z_n)$ and $i\colon Z_n \to X$ the closed immersion associated to the ideal sheaf $\mathcal{I}^n$. From \cite[Lemma 3.3]{Thomason:1997}, we know $\mathbf{L} i^\ast P\in \operatorname{Perf}(Z_n)$ has full support. The induction hypothesis ensures that $\mathbf{R} i_\ast D^b_{\operatorname{coh}}(Z_n)$ belongs to $\mathbf{R} i_\ast \mathcal{S}^\prime$ because $D^b_{\operatorname{coh}}(Z_n)= \langle \mathcal{S}^\prime \rangle$ where $\mathcal{S}^\prime$ is the collection consisting of $\mathbf{L} i^\ast P$ and objects of the form $\mathbf{R} t_\ast \mathbf{L} (i\circ t)^\ast P$ where $t \colon Z^\prime \to Z_n$ is a closed immersion from an integral scheme. Note that $\mathbf{R} i_\ast \mathcal{S}^\prime$ is contained in $\mathcal{S}$, and so, it follows that $E$ belongs to $\langle \mathcal{S} \rangle$ as claimed. 
    
    We will now consider the case where $E$ has full support. Moreover, as objects in $D^b_{\operatorname{coh}}(X)$ are finitely built by their cohomology sheaves, we can impose that $E$ is a coherent $\mathcal{O}_X$-module. First, let us consider the case where $X$ is integral. We have that $E$ is a coherent $\mathcal{O}_X$-module with full support on the integral scheme $X$. By \Cref{lem:derived_generic_free}, there exists a dense open immersion $j \colon U \to X$ such that $\mathbf{L} j^\ast E \in \langle \mathbf{L} j^\ast P \rangle_1$ and $\mathbf{L} j^\ast G \in \langle \mathbf{L} j^\ast E \rangle_1$ in $D^b_{\operatorname{coh}}(U)$. If $U=X$, then we are done, so assume $U$ is a properly contained subset of $X$. Note that $\mathbf{L} j^\ast E$ is a direct summand of a bounded complex of the form $\oplus_{n\in \mathbb{Z}} \mathbf{L} j^\ast P^{\oplus r_n} [n]$, so we have an object $\mathbf{L} j^\ast A$ and an isomorphism of bounded complexes $\phi \colon \mathbf{L} j^\ast A \oplus \mathbf{L} j^\ast E \to \oplus_{n\in \mathbb{Z}} \mathbf{L} j^\ast P^{\oplus r_n} [n]$ in $D^b_{\operatorname{coh}}(U)$. There exists a roof in $D^b_{\operatorname{coh}}(X)$:
    \begin{displaymath}
        \begin{tikzcd}
            & B \\
            {\bigoplus_{n\in \mathbb{Z}} P^{\oplus r_n} [n]} && {A \oplus E}.
            \arrow["f", from=1-2, to=2-1]
            \arrow["g"', from=1-2, to=2-3]
        \end{tikzcd}
    \end{displaymath}
    Observe that the cones of $f,g$ are supported on the properly contained closed subset $Z$ of $X$ as the map $\mathbf{L} j^\ast \colon D^b_{\operatorname{coh}}(X) \to D^b_{\operatorname{coh}}(U)$ is a Verdier localization. This ensures the objects $\operatorname{cone}(f), \operatorname{cone}(g)$ are in $\mathcal{S}$ by our inductive hypothesis. Therefore, $B$ belongs to $\mathcal{S}$, and hence, so does $A\oplus E$.
    
    Lastly, we will consider the case where $X$ is not integral. Denote the irreducible components of $X=\operatorname{Supp}(E)$ by $Z_1,\ldots, Z_n$. From \cite[\href{https://stacks.math.columbia.edu/tag/01YD}{Tag 01YD}]{StacksProject}, there exists a short exact sequence of coherent $\mathcal{O}_X$-modules:
    \begin{displaymath}
        0 \to E_1 \to E \to E_1^\prime \to 0
    \end{displaymath}
    where $\operatorname{Supp}(E_1)\subseteq Z_1$ and $\operatorname{Supp}(E_1^\prime) \subseteq \cup^n_{i=2} Z_i$. Both objects $E_1$ and $E_1^\prime$ are supported on properly contained closed subsets of $X$, and so, each belong to $\langle \mathcal{S} \rangle$ according to the induction hypothesis. However, this implies that $E$ belongs to $\langle \mathcal{S} \rangle$, which completes the proof.
\end{proof}

\begin{theorem}\label{thm:devissage}
    Let $X$ be a Noetherian scheme. Suppose $P$ is a classical generator for $\operatorname{Perf}(X)$. Then $D^b_{\operatorname{coh}}(X)$ is classically generated by $P$ and the objects of the form $\mathbf{R} i_\ast \mathbf{L} i^\ast P$ where $i\colon Z \to X$ is a closed immersion from an integral scheme such that $i(Z)$ is contained in $\operatorname{sing}(X)$.
\end{theorem}

\begin{proof}
    Denote by $\mathcal{U}$ for the collection of objects consisting of $P$ and those of the form $\mathbf{R} i_\ast \mathbf{L} i^\ast P$ where $i\colon Z \to X$ a closed immersion from an integral scheme such that $i(Z)$ is contained in $\operatorname{sing}(X)$. We will show that $\langle \mathcal{U} \rangle = D^b_{\operatorname{coh}}(X)$. By \cite[\href{https://stacks.math.columbia.edu/tag/0BQT}{Tag 0BQT}]{StacksProject}, we can see that $\mathcal{O}_Z$ is finitely built by $\mathbf{L} t^\ast P$ for each closed immersion $t\colon Z \to X$. So $\mathbf{R} t_\ast \mathcal{O}_Z$ is finitely built by $\mathbf{R} t_\ast \mathbf{L} t^\ast P$. This ensures that $\mathbf{R} i_\ast \mathcal{O}_Z$ is contained in $\langle \mathcal{U} \rangle$ for each closed immersion from an integral scheme such that $i(Z)$ is contained in $\operatorname{sing}(X)$. 
    
    Consider an affine open cover $X=\bigcup^n_{i=1} U_i$ with $s_i \colon U_i \to X$ being the associated open immersions. Note that for each $1\leq i \leq n$, one has $\langle \mathbf{L} s^\ast_i \mathcal{U} \rangle$ containing all objects of the form $p_\ast \mathcal{O}_W$ where $p\colon W \to U_i$ is a closed immersion where $W$ such that $p(W)$ is contained in $\operatorname{sing}(U_i)$, including the structure sheaf. From \cite[Corollary 4.3.(1)]{Takahashi:2014}, it follows that $\langle \mathbf{L} s^\ast_i \mathcal{U} \rangle = D^b_{\operatorname{coh}}(U_i)$ because $U_i$ is affine for each $1\leq i \leq n$. 
    
    Let $E$ be an object of $D^b_{\operatorname{coh}}(X)$. There is an object $G_i$ in $\langle \mathcal{U} \rangle_1$ such that $\mathbf{L} s^\ast_i E$ is finitely built by $\mathbf{L} s^\ast_i G_i$ in $D^b_{\operatorname{coh}}(U_i)$ because $\langle \mathbf{L} s^\ast_i \mathcal{U} \rangle = D^b_{\operatorname{coh}}(U_i)$. Set $G=\oplus^n_{j=1} G_i$. It follows that $\mathbf{L} s^\ast_j E$ is finitely built by $\mathbf{L} s^\ast_j G$ in $D^b_{\operatorname{coh}}(U_j)$ for each $1\leq j \leq n$. This ensures that $E_p$ is finitely built by $G_p$ in $D^b_{\operatorname{coh}}(\mathcal{O}_{X,p})$ for each $p$ in $X$. Hence, $E_p \overset{\mathbf{L}}{\otimes} K(p)$ is finitely built by $G_p$ in $D^b_{\operatorname{coh}}(\mathcal{O}_{X,p})$ for each $p$ in $X$ where $K(p)$ is the Koszul complex on a minimal set of generators for the maximal ideal of the local ring $\mathcal{O}_{X,p}$. Indeed, $K(p)$ is finitely built by $\mathcal{O}_{X,p}$, and tensoring with $E_p$ tells us that $E_p \overset{\mathbf{L}}{\otimes} K(p)$ is finitely built by $E_p$.
    
    By \cite[Theorem 1.7]{BILMP:2023}, if $E_p \overset{\mathbf{L}}{\otimes} K(p)$ is finitely built by $G_p$ in $D^b_{\operatorname{coh}}(\mathcal{O}_{X,p})$ for each $p$ in $X$, then $E$ is finitely built by $P\overset{\mathbf{L}}{\otimes}G$ in $D^b_{\operatorname{coh}}(X)$. As $G$ belongs to $\langle \mathcal{U} \rangle_1$, it is a direct summand of an object of the form $\oplus_{l\in \mathbb{Z}} (\mathbf{R} i_{l,\ast} \mathbf{L} i^\ast_l P )^{\oplus r_l}[l]$. By tensoring with $P$ and using projection formula, it follows that $P \overset{\mathbf{L}}{\otimes} \big(\oplus_{l\in \mathbb{Z}} ( \mathbf{R} i_{l,\ast} \mathbf{L} i^\ast_l P )^{\oplus r_l}[l])$ is isomorphic to $\oplus_{l\in \mathbb{Z}} ( \mathbf{R} i_{l,\ast} \mathbf{L} i^\ast_l (P\overset{\mathbf{L}}{\otimes}P ))^{\oplus r_l}[l]$. Note that $P\overset{\mathbf{L}}{\otimes} G$ is a direct summand of $\oplus_{l\in \mathbb{Z}} (\mathbf{R} i_{l,\ast} \mathbf{L} i^\ast_l (P\overset{\mathbf{L}}{\otimes}P ))^{\oplus r_l}[l]$. As $P\overset{\mathbf{L}}{\otimes}P$ is a classical generator for $\operatorname{Perf}(X)$, we see that $\oplus_{l\in \mathbb{Z}} (\mathbf{R} i_{l,\ast} \mathbf{L} i^\ast_l (P\overset{\mathbf{L}}{\otimes}P ))^{\oplus r_l}[l]$ is finitely built by $\oplus_{l\in \mathbb{Z}} (\mathbf{R} i_{l,\ast} \mathbf{L} i^\ast_l P)^{\oplus r_l}[l]$. Therefore, $E$ must be contained in $\langle \mathcal{U} \rangle$.
\end{proof}

\begin{remark}\label{rmk:compare_proofs_singular}
    It is important to assume $P$ is a classical generator for $\operatorname{Perf}(X)$ in \Cref{thm:devissage}. A possible relaxation might be imposing $P$ be a perfect complex on $X$ with full support, but there are counterexamples. For example, $P=\mathcal{O}_X$ where $X=\mathbb{P}^1_k$ for $k$ a field because $\mathcal{O}_X$ is not a strong generator for $D^b_{\operatorname{coh}}(\mathbb{P}^1_k)$ and the singular locus is empty. 
\end{remark}

A useful application of \Cref{thm:devissage} is when the singular locus is finite.

\begin{example}\label{ex:surface_generator}
    Let $X$ be an integral Noetherian scheme with at worst finitely many isolated singularities. Then the singular locus of $X$ is finite, say with closed points $p_1,\ldots,p_n$ and associated closed immersions $\pi_i \colon \operatorname{Spec}(\kappa(p_i)) \to X$. Suppose $P\in \operatorname{Perf}(X)$ is a classical generator for $\operatorname{Perf}(X)$. Then 
    \begin{displaymath}
        P \oplus (\oplus^n_{i=1} \pi_{i,\ast} \mathcal{O}_{\operatorname{Spec}(\kappa(p_i))})
    \end{displaymath}
    is a classical generator for $D^b_{\operatorname{coh}}(X)$. There are two cases where $P$ can be made more precise:
    \begin{enumerate}
        \item If $X$ is a quasi-projective variety over a field with very ample line bundle $\mathcal{L}$, then one can take $P=\oplus^{\dim X}_{i=0} \mathcal{L}^{\otimes i}$ (see e.g.\ \cite[Theorem 4]{Orlov:2009}).
        \item Suppose $X$ is a scheme that is projective over a Noetherian ring $R$. Denote the closed immersion by $t\colon X \to \mathbb{P}^N_R$. Set $\mathcal{L}$ to be the line bundle $t^\ast \mathcal{O}_{\mathbb{P}^n_R} (-1)$. By \cite[\href{https://stacks.math.columbia.edu/tag/0A9V}{Tag 0A9V}]{StacksProject}, we know that $\oplus^N_{i=0} \mathcal{O}_{\mathbb{P}^n_R} (-1)$ is a classical generator for $\operatorname{Perf}(\mathbb{P}^n_R)$. However, the closed immersion $i$ is an affine morphism, so \cite[\href{https://stacks.math.columbia.edu/tag/0BQT}{Tag0BQT}]{StacksProject} tells us that $i^\ast (\oplus^N_{i=0} \mathcal{O}_{\mathbb{P}^n_R} (-1))$ is a classical generator for $\operatorname{Perf}(X)$.
    \end{enumerate}
\end{example}

Now we work towards the rest. Initially, we showed that $\langle \mathbf{R}\pi_\ast D^b_{\operatorname{coh}}(Y) \rangle = D^b_{\operatorname{coh}}(X)$ for $\pi\colon Y \to X$ a proper surjective morphism by Noetherian induction coupled with \Cref{thm:devissage}. But we later realized how to refine this.

\begin{lemma}\label{lem:boundedness}
    Let $\pi \colon Y \to X$ be a proper surjective morphism of Noetherian schemes. For every object $K$ in $\mathbf{R}\pi_\ast D^{-}_{\operatorname{coh}}(Y)$ and integer $a$, there is a map $f \colon K \to L$ with $L$ in $\mathbf{R}\pi_\ast D^b_{\operatorname{coh}}(Y)$ such that $\operatorname{cone}(f)$ is in $D^{\leq a}_{\operatorname{Qcoh}}(X)$.
\end{lemma}

\begin{proof}
    Let $K:=\mathbf{R}\pi_\ast E$. Then $L=\mathbf{R}\pi_\ast (\tau_{\geq a-N} E)$ fits the bill for some integer $N$.
\end{proof}

\begin{lemma}\label{lem:factoring_n_fold_extensions}
    Let $\pi \colon Y \to X$ be a proper surjective morphism of Noetherian schemes. Let $n\geq 0$. For each $K$ in $D^{-}_{\operatorname{coh}}(X)$ which belongs to $\big(\mathbf{R}\pi_\ast D^{-}_{\operatorname{coh}}(Y) \big)^{\star n}$ and integer $a$, there is a map $f\colon K \to L$ with $L$ an object of $\big(\mathbf{R}\pi_\ast D^b_{\operatorname{coh}}(Y) \big)^{\star n}$ such that $\operatorname{cone}(f)$ in $D^{<a}_{\operatorname{Qcoh}}(X)$.
\end{lemma}

\begin{proof}
    The case where $n=0$ follows from \Cref{lem:boundedness}. We check the case for $n=1$, leaving the general case omitted as it follows similarly. There is a distinguished triangle in $D_{\operatorname{Qcoh}}(X)$:
    \begin{displaymath}
        A \to B \to K \to A[1]
    \end{displaymath}
    where $A,B$ are objects of $\mathbf{R}\pi_\ast D^{-}_{\operatorname{coh}}(Y)$. Then \Cref{lem:boundedness} ensures there is a commutative square:
    \begin{displaymath}
        \begin{tikzcd}[ampersand replacement=\&]
            A \& B \\
            {A^\prime} \& {B^\prime}
            \arrow[from=1-1, to=1-2]
            \arrow[from=1-1, to=2-1]
            \arrow[from=1-2, to=2-2]
            \arrow[from=2-1, to=2-2]
        \end{tikzcd}
    \end{displaymath}
    where the vertical maps have cones in $D^{<a-1}_{\operatorname{Qcoh}}(X)$. This can be accomplished by choosing a map $B \to B^\prime$ via \Cref{lem:boundedness}, and a map $A\to A^\prime$ a sufficiently good approximation for which $A \to B^\prime$ factors through $A^\prime$. Now extending the square above to a morphism of distinguished triangles, the base case follows.
\end{proof}

\begin{lemma}\label{lem:truncate_down_to_bounded}
    Let $\pi \colon Y \to X$ be a proper surjective morphism of Noetherian schemes. If $E$ is an object of $D^b_{\operatorname{coh}}(X)\cap \langle \mathbf{R}\pi_\ast D^{-}_{\operatorname{coh}}(Y) \rangle_n$, then $E$ belongs to $\langle \mathbf{R}\pi_\ast D^b_{\operatorname{coh}}(Y) \rangle_n$. 
\end{lemma}

\begin{proof}
    The hypothesis tells us $E$ is a retract of an object $K\in \big(\mathbf{R}\pi_\ast D^{-}_{\operatorname{coh}}(Y) \big)^{\star n}$. Choose an integer $a$ such that $E$ belongs to $D^{> a}_{\operatorname{Qcoh}}(X)$. Consider the map $K \to L$ obtained from \Cref{lem:factoring_n_fold_extensions}. We can check that $E$ is a retract of $L$ as the retraction $K \to E$ factors through $L$ for degree reasons.
\end{proof}

\begin{lemma}\label{lem:weak_form_aoki}
    If $\pi \colon Y \to X$ is a proper surjective morphism of Noetherian schemes, then there is an integer $n$ such that $D^b_{\operatorname{coh}}(X)= \langle \mathbf{R}\pi_\ast D^b_{\operatorname{coh}}(Y) \rangle_n$.
\end{lemma}

\begin{proof}
    There is an integer $n$ such that each object of $D^-_{\operatorname{coh}}(X)$ belongs to $\langle \mathbf{R} \pi_\ast D^{-}_{\operatorname{coh}}(Y) \rangle_n$. This can be observed from the argument for \cite[Proposition 4.4]{Aoki:2021}. The desired claim follows from \Cref{lem:truncate_down_to_bounded}.
\end{proof}

We obtain generalizations of \cite[Theorem E]{Lank:2023}, \cite[Lemma 3.15]{Lank/Olander:2024}, and \cite[Corollary 8.1.3]{Gaitsgory:2013}.

\begin{proposition}\label{prop:devissage_classical}
    Let $\pi\colon Y \to X$ be a proper surjective morphism of Noetherian schemes. If $G$ is a classical (resp.\ strong) generator for $D^b_{\operatorname{coh}}(Y)$, then $\mathbf{R}\pi_\ast G$ is a classical (resp.\ strong) generator for $D^b_{\operatorname{coh}}(X)$.
\end{proposition}

\begin{proof}
    We prove the case for classical generators as the other case follows similarly. By \Cref{lem:weak_form_aoki}, there is an integer $n$ such that $D^b_{\operatorname{coh}}(X)= \langle \mathbf{R}\pi_\ast D^b_{\operatorname{coh}}(Y) \rangle_n$. Let $E$ be an object of $D^b_{\operatorname{coh}}(X)$. We can find an object $E^\prime$ in $D^b_{\operatorname{coh}}(Y)$ such that $E$ is finitely built by $\mathbf{R}\pi_\ast E^\prime$. However, $\mathbf{R}\pi_\ast G$ finitely builds $\mathbf{R}\pi_\ast E^\prime$, and so, $E$ belongs to $\langle D^b_{\operatorname{coh}}(X) \rangle$. This completes the proof.
\end{proof}

\begin{example}\label{ex:generators_variety}
    Let $X$ be a variety over a perfect field. From \cite[Theorem 4.1]{deJong:1996}, there exists a proper surjective morphism $\pi \colon \widetilde{X}\to X$ such that $\dim \widetilde{X} = \dim X$ and $\widetilde{X}$ is a regular quasi-projective variety. Let $\mathcal{L}$ be a very ample line bundle on $\widetilde{X}$. By \cite[Theorem 4]{Orlov:2009} and \cite[Proposition 7.31]{Rouquier:2008}, we know that $G:= \bigoplus^{\dim X}_{i=0} \mathcal{L}^{\otimes i}$ is a strong generator for $D^b_{\operatorname{coh}}(\widetilde{X})$. From \Cref{prop:devissage_classical}, one has that $\mathbf{R}\pi_\ast G$ is a strong generator for $D^b_{\operatorname{coh}}(X)$.
\end{example} 

\bibliographystyle{alpha}
\bibliography{mainbib}

\end{document}